\documentclass[12pt]{amsart}
\usepackage{amsfonts}
\usepackage{eurosym}
\usepackage{amssymb}
\usepackage{dsfont}
\usepackage{color}
\usepackage{graphicx}

\newtheorem{thm}{Theorem}

\newtheorem{lem}[thm]{Lemma}

\newtheorem{cor}[thm]{Corollary}
\theoremstyle{definition}
\newtheorem{df}[thm]{Definition}

\newtheorem{rem}[thm]{Remark}

\newtheorem{con}[thm]{Conjecture}
\newtheorem{problem}[thm]{Problem}
\theoremstyle{remark}
\email{artur.bartoszewicz@p.lodz.pl}
\email{malfil@math.uni.lodz.pl}
\email{szymon.glab@p.lodz.pl}
\email{wisniows@univ.szczecin.pl}
\thanks{The last author has been supported by the NCN Grant (``Diamond grant") No. 0168/DIA/2014/43.}
\email {jswaczyna@wp.pl}

\begin{document}
\title{On generating regular Cantorvals connected with geometric Cantor sets}
\date{\today }
\author[A. Bartoszewicz]{Artur Bartoszewicz}
\address{Institute of Mathematics, \L \'od\'z University of Technology, ul.
W\'olcza\'nska 215, 93-005 \L \'od\'z, Poland}
\author[M. Filipczak]{Ma\l gorzata Filipczak}
\address{Faculty of Mathematics and Computer Science, \L \'od\'z University,
ul. Stefana Banacha 22, 90-238 \L \'od\'z, Poland}
\author[S. G\l \c{a}b]{Szymon G\l \c{a}b}
\address{Institute of Mathematics, \L \'od\'z University of Technology, ul.
W\'olcza\'nska 215, 93-005 \L \'od\'z, Poland}
\author[F. Prus-Wi\'sniowski]{Franciszek Prus-Wi\'sniowski}
\address{Institute of Mathematics, University of Szczecin, ul. Wielkopolska
15, 70-453 Szczecin, Poland}
\author[J. Swaczyna]{Jaros\l aw Swaczyna}
\address{Institute of Mathematics, \L \'od\'z University of Technology, ul.
W\'olcza\'nska 215, 93-005 \L \'od\'z, Poland}
\date{}

\begin{abstract}
We show that the Cantorvals connected with the geometric Cantor sets are not
achievement sets of any series. However many of them are attractors of IFS
consisting of affine functions.
\end{abstract}

\maketitle

\section{Introduction}

Let $(x_{n})=(x_{1},x_{2},\ldots )$ be an absolutely summable sequence, in
symbols $(x_{n})\in \ell _{1}$, and let 
\begin{equation*}
E(x_{n})=\left\{ \sum_{n=1}^{\infty }\varepsilon _{n}x_{n}:(\varepsilon
_{n})\in \{0,1\}^{\mathbb{N}}\right\}
\end{equation*}%
denote the set of all subsums of the series $\sum_{n=1}^{\infty }x_{n}$. The
set $E(x_{n})$ is called the \emph{achievement set of $(x_{n})$} (see \cite%
{J}). It is easy to see that for $(x_{n})=(\frac{2}{3},\frac{2}{3^{2}}%
,\ldots )$ the set $E(x_{n})$ is equal to the classic Cantor ternary set $C$
and for $(x_{n})=(\frac{1}{2},\frac{1}{2^{2}},\ldots )$ we have $%
E(x_{n})=[0,1]$. Achievement sets of sequences have been considered by many
authors, some results have been rediscovered several times. The following
properties of sets $E(x_{n})$ were described by Kakeya in \cite{K} in 1914:

\begin{itemize}
\item[(i)] $E(x_n)$ is a compact perfect set,

\item[(ii)] If $|x_{n}|>\sum_{i>n}|x_{i}|$ for all sufficiently large $n$'s,
then $E(x_{n})$ is homeomorphic to the Cantor set $C$,

\item[(iii)] If $|x_{n}|\leq \sum_{i>n}|x_{i}|$ for all sufficiently large $%
n $'s, then $E(x_{n})$ is a finite union of closed intervals. Moreover, if $%
|x_{n}|\geq |x_{n+1}|$ for all but finitely many $n$'s and $E(x_{n})$ is a
finite union of closed intervals, then $|x_{n}|\leq \sum_{i>n}|x_{i}|$ for
all but finitely many $n$'s.
\end{itemize}

One can see that $E(x_{n})$ is finite if and only if $x_{n}=0$ for all but
finite number of $n$'s, i.e. $(x_{n})\in c_{00}$. Kakeya conjectured that if 
$(x_{n})\in \ell _{1}\setminus c_{00}$, then $E(a_{n})$ is always
homeomorphic to the Cantor set $C$ or it is a finite union of intervals. In
1980 Weinstein and Shapiro in \cite{WS} gave an example whose showed that
the Kakeya hypothesis is false. Independently, the similar example was given
by Ferens in \cite{F}. In \cite{GN} Guthrie and Nymann gave a very simple
example of a sequence whose achievement set is not a finite union of closed
intervals but it has nonempty interior. They used the following sequence $%
(t_{n})=(\frac{3}{4},\frac{2}{4},\frac{3}{16},\frac{2}{16},\ldots )$.
Moreover, they formulated the following:

\begin{thm}
\label{char} For any $(x_n) \in \ell_1 \setminus c_{00}$, the set $E(x_n)$
is one of the following types:

\begin{itemize}
\item[(i)] a finite union of closed intervals,

\item[(ii)] homeomorphic to the Cantor set $C$,

\item[(iii)] homeomorphic to the set $\mathbf{T}=E(t_n)=E(\frac{3}{4}, \frac{%
2}{4}, \frac{3}{16}, \frac{2}{16}, \frac{3}{64}, \ldots)$.
\end{itemize}
\end{thm}

Although their proof had a gap, the theorem is true and the correct proof
was given by Nymann and Saenz in \cite{NS0}. Guthrie, Nymann and Saenz have
observed that the set $\mathbf{T}$ is homeomorphic to the set $N$ described
by the formula 
\begin{equation*}
N=[0,1]\setminus \bigcup_{n\in \mathbb{N}}U_{2n},
\end{equation*}%
where $U_{n}$ denotes the union of $2^{n-1}$ open middle thirds which are
removed from the interval $[0,1]$ at the $n$-th step in the construction of
the classic Cantor ternary set $C$. Such sets are called Cantorvals in the
literature (to emphasize the similarity to the interval and to the Cantor
set simultaneously). It is known that a Cantorval is such nonempty compact
set in $\mathbb{R}$, that it is the closure of its interior and both
endpoints of any nontrivial component are accumulation points of its trivial
components. Other topological characterizations of Cantorvals can be found
in \cite{BFPW} and \cite{MO}. Let us observe that Theorem \ref{char} states
that $\ell _{1}$ can be divided into four disjoint sets $c_{00},\ \mathcal{C}%
,\ \mathcal{J}$ and $\mathcal{MC}$, where $\mathcal{J}$ consists of
sequences with achievement sets being finite unions of intervals, $\mathcal{C%
}$ consists of sequences $(x_{n})$ with $E(x_{n})$ homeomorphic to the
Cantor set and $\mathcal{MC}$ consists of the sequences with achievement
sets homeomorphic to the set $\mathbf{T}$ or equivalently to $N$ ($\mathcal{%
MC}$ is the abbrevation of \emph{Middle-Cantorval} because the structure of
the achievement set should be symmetric). Algebraic and topological
properties of these subsets of $\ell _{1}$ have been recently studied in 
\cite{BBGS}. All known examples of sequences which achievement sets are
Cantorvals belong to the class of multigeometric sequences. This class was
deeply investigated in \cite{J}, \cite{BFS} and \cite{BBFS}. In particular,
the achievement sets of multigeometric series and sets obtained in more
general case are the attractors of the affine iterated function systems, see 
\cite{BBFS}. More information on achievement sets can be found in the
surveys \cite{BFPW}, \cite{N1} and \cite{N2}. The aim of our paper is to
show that the most known Cantorval $N=[0,1]\setminus \bigcup_{n\in \mathbb{N}%
}U_{2n}=C\cup \bigcup_{n\in \mathbb{N}}U_{2n-1}$, and the other members of
large class of Cantorvals connected with geometric Cantor sets are not
isometric to an achievement set for any sequence (section 2) but they are
attractors of iterated function system (IFS) consisting of several affine
functions (section 3).

It is almost obvious that any achievement set $E$ of a summable sequence
contains zero and is symmetric in the sense that there exists a number $t$
such that if $t-x\in E$ then $t+x\in E$ too. It is a natural question if
every set with these properties is an achievement set for some sequence.
This question was posted by W. Kubi\'{s} in \L \'{o}d\'{z} in 2015.In
particular, in \cite{BPW} the authors ask if the Cantorval $N$ is an
achievement set of any sequence. 

Recently, the authors of \cite{BPW} have considered the N-G-Cantorval $%
\mathbf{T}$ (defined in Theorem \ref{char} (iii)), the Cantor set $%
Y=\partial \mathbf{T}$ and the Cantorval $Z$ whose intervals are the gaps of 
$\mathbf{T}$, and the gaps of $Z$ are the intervals of $\mathbf{T}$. They
introduced the notion of center of distances and used it to prove that $Y$
and $Z$ are not achievement sets for any series. Recall that the center of
distances of a metric space $(X,d)$ is the set 
\begin{equation*}
C(X)=\{\alpha \geq 0:\forall x\in X\exists y\in X\text{ , such that }%
d(x,y)=\alpha \}.
\end{equation*}%
Clearly for $X=E(x_{n})$ for a sequence $(x_{n})$ of positive terms we have $%
\{x_{n}:n\in \mathbb{N}\}\subset C(X)$. If we assume that $N=E(x_{n})$, we
conclude by the Second Gap Lemma (Lemma \ref{second_gap_lemma} in the second
section) that $x_{n}=\frac{2}{9}$ for some $n\in \mathbb{N}$. But for $x=%
\frac{25}{54}$ we observe that $x\in N$\ and\ the set $\{y\in N:|x-y|=\frac{2%
}{9}\}=\emptyset $. Therefore $N$ is not an achievement set. We are not able
to use this method to prove our main theorem but the result itself has a
similar nature. Our result and that from \cite{BPW} suggest the following:

\begin{con}
Let $X$ be a Cantor set and $Y,Z$ be Cantorvals with $X=\partial Y=\partial
Z $. Then at most one set of $X,Y,Z$ can be an achievement set.
\end{con}

\section{Main results}

Let us assume that $(x_{n})$ is a nonincreasing, absolutely summable
sequence of positive real numbers. Denote (as in \cite{GN}, \cite{NS0}, \cite%
{BFPW}): 
\begin{equation*}
E=E(x_{n})=\left\{ \sum_{n=1}^{\infty }\varepsilon _{n}x_{n}:(\varepsilon
_{n})\in \{0,1\}^{\mathbb{N}}\right\} \text{;}
\end{equation*}%
\begin{equation*}
E_{k}=\left\{ \sum_{n=k+1}^{\infty }\varepsilon _{n}x_{n}\right\} \mbox{;}\,%
\text{\ }F_{k}=\left\{ \sum_{n=1}^{k}\varepsilon _{n}x_{n}\right\} \text{.}
\end{equation*}%
Of course we have 
\begin{equation*}
E=E_{k}+F_{k}=\bigcup_{f\in F_{k}}(f+E_{k})
\end{equation*}%
\begin{equation*}
E_{k}=E_{k+1}\cup (x_{k+1}+E_{k+1})
\end{equation*}%
for all $k\in \mathbb{N}$. Moreover, let $r_{k}:=\sum_{n=k+1}^{\infty }x_{n}$%
. By a gap in $E$ we understand any such interval $(a,b)$ that $a\in E,\
b\in E$ and $(a,b)\cap E=\emptyset $. The following two lemmas can be found
in \cite{BFPW}. We provide their proofs for reader's convenience. The first
is obvious.

\begin{lem}
(First Gap Lemma) If $x_{k}>r_{k}$ then $(r_{k},x_{k})$ is a gap in $E$.
\end{lem}

The next observation is extracted from the proof of the crucial Lemma 4 of 
\cite{NS0}, where it was formulated as not a quite correct claim (however
the Lemma and the main result of \cite{NS0} are true).

\begin{lem}
\label{second_gap_lemma}(Second Gap Lemma) Let $(a,b)$ be a gap in $E$, and
let $p$ be defined by the formula $p:=\max \{n:x_{n}\geq b-a\}$. Then:

\begin{itemize}
\item[(i)] $b\in F_{p}$,

\item[(ii)] If $F_{p}=\{f_{1}^{(p)}<f_{2}^{(p)}<\ldots <f_{m(p)}^{(p)}\}$
and $b=f_{j}^{(p)}$, then $a=f_{j-1}^{(p)}+r_{p}$.
\end{itemize}
\end{lem}

\begin{proof}
Let us observe that $b\in F_{p}$. If not, we have $b=\sum_{n\in A}x_{n}$,
where $A$ contains some $i$ greater than $p$. Hence $x_{i}<b-a$ and $%
b-x_{i}\in E$. Consequently we obtain the element $b-x_{i}\in E$ which
belongs to the gap $(a,b)$, contradiction. So let $b=f_{j}^{(p)}$. Then, by
the definition of a gap, we have $a\geq f_{j-1}^{(p)}$. Now we will show
that $f_{j-1}^{(p)}+r_{p}\leq b$. Suppose $f_{j-1}^{(p)}+r_{p}>b$. Let us
consider the sequence $(f_{j-1}^{(p)}+r_{i})_{i=p}^{\infty }$ which is
decreasing and converging to $f_{j-1}^{(p)}$. Since the difference between
the consecutive terms of this sequence is smaller than $b-a$ (by the
definition of $p$), there exists a term which belongs to $(a,b)$, which
gives a contradiction. Consequently $f_{j-1}^{(p)}+r_{p}\leq a$. On the
other hand $a<f_{i}^{(p)}$ for $i\geq j$ and hence $a\in
\bigcup_{i<j}(f_{i}^{(p)}+E_{p})$. Therefore $a\leq \sup
(f_{j-1}^{(p)}+E_{p})=f_{j-1}^{(p)}+r_{p}$.
\end{proof}

For our purpose we need the following strenghtening of the Second Gap Lemma.

\begin{lem}
\label{lemat3} Suppose that $(a,b)$ is a gap in $E=E(x_{n})$ such that for
any gap $\left( a_{1},b_{1}\right) $ with $b_{1}<a$ we have $b-a>b_{1}-a_{1}$
(in other words $\left( a,b\right) $ is the longest gap from the left). Then 
$b=x_{k}$ for some $k\in \mathbb{N}$ and $a=r_{k}$.
\end{lem}

\begin{proof}
By the Second Gap Lemma $b$ is a finite sum of terms of $\left( x_{n}\right) 
$. Let $b=x_{n_{1}}+...+x_{n_{m}}$\ with $x_{n_{1}}\geq ...\geq x_{n_{m}}$.
Suppose that $m\geq 2$.\ Firstly observe that $x_{n_{m}}\geq b-a$ (indeed,
if $x_{n_{m}}<b-a$ then $b-x_{n_{m}}\in \left( a,b\right) \cap E$ which is
impossible). Of course $x_{n_{m}}<b$ and, since $\left( a,b\right) $ is a
gap, $x_{n_{m}}\leq a$. Any gap in the set $X:=E\cap \left[ 0,x_{n_{m}}%
\right] $ is shorter than $b-a$. On the other hand, $b\in X+\left(
b-x_{n_{m}}\right) $ and $X+\left( b-x_{n_{m}}\right) \subset E$, so $\left(
a,b\right) \cap \left( X+\left( b-x_{n_{m}}\right) \right) =\emptyset $, a
contradiction. Thus $m=1$ which means that $b=x_{k}$ for some $k\in \mathbb{N%
}$.\newline
Since $a\in E$, $r_{k}\geq a$. Suppose that $r_{k}>a$. Let $m$ be the
smallest number satisfying $\sum_{n=k+1}^{m}x_{n}>a$. Hence $%
\sum_{n=k+1}^{m}x_{n}>b$, because $\left( a,b\right) $ is a gap. Let now $%
X:=E\cap \left[ 0,x_{m}\right] $. Then the set $X+\sum_{n=k+1}^{m-1}x_{n}$
is included in $E$ and has allgaps shorter than $b-a$, which gives a
contradiction again.
\end{proof}

Note that for $q\geq 1/2$ the achievement set $E\left( q^{n}\right) $\ is an
interval. For $q\in (0,1/2)$ let $C_{q}=E(\left( 1-q\right) q^{n-1})$. Then $%
C_{q}=[0,1]\setminus \bigcup_{n=1}^{\infty }U_{n}^{q}$ where $U_{n}^{q}$ is
a union of $2^{n-1}$ open intervals, each of the length $q^{n-1}(1-2q)$,
removed from $[0,1]$ in the $n$-th step in $C_{q}$ construction, exactly as
in the construction of the classic Cantor set $C=C_{1/3}$.

Now, let us divide $\mathbb{N}$ into two infinite sets: $H$ and its
complement $H^{c}$. Consider the set $N_{q}^{H}=[0,1]\setminus \bigcup_{n\in
H}U_{n}^{q}$. $N_{q}^{H}$ is evidently a Cantorval and it is regular in the
sense that in each step of the construction we remove or we leave all the
components of the set $U_{n}^{q}$. Let us observe that the Cantor set $C_{q}$
is the boundary of any Cantorval $N_{q}^{H}$. For $H=2\mathbb{N}$ we write
simply $N_{q}$ instead of $N_{q}^{2\mathbb{N}}$. Note that $N=N_{1/3}$ is
the generic example of Cantorval considered by Guthrie and Nymann which was
proved to be homeomorphic with $\mathbf{T}=E(\frac{3}{4},\frac{2}{4},\frac{3%
}{16},\frac{2}{16},\frac{3}{64},\ldots )$.\ Thus the sets $N_{q}^{H}$ form a
wide class of Cantorvals containing the very important example. In the
following figure is presented the nested construction of the set $N_{q}$,
for some $q\leq 1/3$.

\begin{center}
\includegraphics[height=4cm,width=16cm]{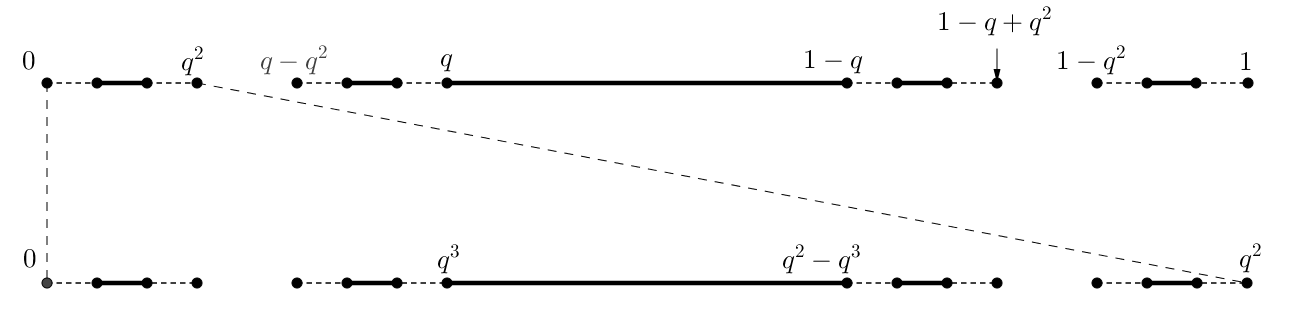}\\[0pt]
Figure 1
\end{center}

In the next lemma we will consider the longest interval which is contained
in the finite union of translations of $N_{q}^{H}$. It is easy to see that
if $1\notin H$, $J=[q,1-q]$ and $t=1-2q$, then the longest interval $I$
which is contained in $N_{q}^{H}\cup (t+N_{q}^{H})$ might have length $%
2\left\vert J\right\vert $. However, it can be a bit longer. For example, if 
$(t+N_{q}^{H})$\ contains the component $\left[ a,b\right] $\ with $a<q\leq
b $\ then $\inf I<\inf J$. In the analogous way $\sup I$\ might be bigger
than $\sup (t+J)$. What is even more interesting, there are such $q$'s for
which $\sup J<\inf \left( t+J\right) $\ and $J\cup \left( t+J\right) \subset
I$\ - see Figure 2.

\begin{center}
\includegraphics[height=3cm,width=15cm]{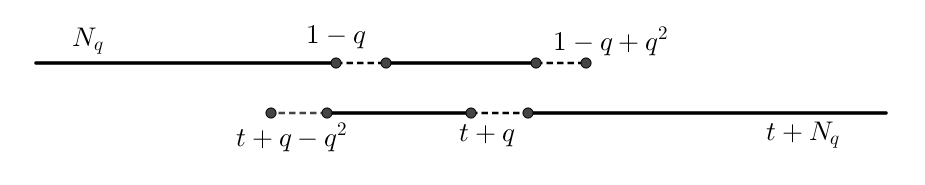}\\[0pt]
Figure 2
\end{center}

Such a situation can be observed, for example, for $q=1/4$. One can see that 
$N_{\frac{1}{4}}\cup (N_{\frac{1}{4}}+\frac{35}{64})$ contains an interval $%
I $ with $\left\vert I\right\vert >2\left\vert J\right\vert $. In fact $%
|I|=2\cdot \frac{1}{2}+\frac{3}{64}$.

\begin{lem}
\label{MainLemma} Let $N_{q}^{H}$ be a regular Cantorval, $1\notin H$, $2\in
H$. \newline
(i) If $q\leq 1/3$ and $0=t_{1}<t_{2}<\dots <t_{l}$ for some $l\in \mathbb{N}
$, then the length of the longest interval contained in the union of
translations $\bigcup_{i=1}^{l}(t_{i}+N_{q}^{H})$ is smaller than $%
(l(1-2q)+(l-1)q^{2})$. \newline
(ii) If $q\in \left( 1/3,\sqrt{2}-1\right) $\ and $t>0$\ then the length of
the longest interval contained in $N_{q}^{H}\cup (t+N_{q}^{H})$\ is smaller
than $5(1-2q)$.\newline
(iii) If $q\in \left( \sqrt{2}-1,1/2\right) $\ and $t>0$\ then the length of
the longest interval contained in $N_{q}^{H}\cup (t+N_{q}^{H})$\ is smaller
than $4(1-2q)$.
\end{lem}

\begin{proof}
Denote by $J$ the longest sub-interval $[q,1-q]$ of $N_{q}^{H}$ and by $I$
the longest interval contained in $N_{q}^{H}\cup (t+N_{q}^{H})$.

(i) First, fix $t>0$ and consider the union $N_{q}^{H}\cup (t+N_{q}^{H})$.
If $t=1-2q=\left\vert J\right\vert $ then $|I|=2|J|$. If $t<\left\vert
J\right\vert $ then there may exist intervals $K_{1}\subset t+N_{q}^{H}$ and 
$K_{2}\subset N_{q}^{H}$ such that $q\in K_{1}$ and $\inf K_{1}<q$, and $%
t+1-q\in K_{2}$ and $\sup K_{2}<t+1-q$. Since $\left\vert K_{1}\right\vert
\leq q^{3}$ and $\left\vert K_{2}\right\vert \leq q^{3}$, we obtain%
\begin{equation*}
\left\vert I\right\vert <2\left\vert J\right\vert +2q^{3}<2\left\vert
J\right\vert +q^{2}\text{.}
\end{equation*}%
Suppose now that $t>|J|$ and remind that $\left( 1-q-q^{2},1-q^{2}\right) $
is a gap in $N_{q}^{H}$. It follows that if $t-\left\vert J\right\vert
>q^{2} $ then $\left[ 1-q+q^{2},t+q\right] \setminus \left( N_{q}^{H}\cup
(t+N_{q}^{H})\right) \neq \emptyset $. It means that $\left\vert
I\right\vert <2\left\vert J\right\vert +q^{2}$.

Let $0=t_{1}<t_{2}<t_{3}$ and denote by $I$ the largest interval contained
in the $\bigcup_{i=1}^{3}(t_{i}+N_{q}^{H})$ (if there is more then one such
interval, we denote by $I$ the right one). Denote by$\ d$ the number $%
3(1-2q)+2q^{2}$. If $t_{2}<\left\vert J\right\vert +q^{2}$ and $%
t_{3}<t_{2}+\left\vert J\right\vert +q^{2}$ then (in the same manner as in
the previous consideration) we obtain that $\left\vert I\right\vert <d$. If $%
t_{3}>t_{2}+\left\vert J\right\vert +q^{2}$ then $I=\left( t_{3}+J\right) $
or $I\subset N_{q}^{H}\cup (t_{2}+N_{q}^{H})$. In both cases $\left\vert
I\right\vert <d$. Finally, if $t_{2}>\left\vert J\right\vert +q^{2}$ then $%
J\cap I=\emptyset $ or the set $t_{3}+N_{q}^{H}$ "fills the gap" between $J$
and $t_{2}+J$, which is possible only if $t_{3}<\left( 1-q\right) +q^{2}$.
In any case $\left\vert I\right\vert <d$.

In a similar way one can check that for any $l\in N$ and $%
0<t_{1}<t_{2}<\dots <t_{l}$ the length of the largest interval contained in
the union of translations $\bigcup_{i=1}^{l}(t_{i}+N_{q}^{H})$ is not
greater than $l(1-2q)+(l-1)q^{2}$.

(ii) Note that for any $q\in \left( 0,1/2\right) $\ the length of $I$\ is
smaller then $2\left\vert J\right\vert +3q^{2}$. Moreover,\ for $q<\sqrt{2}%
-1 $,\ the interval $J$\ is longer then $q^{2}$. Therefore, for such $q$'s,\ 
$\left\vert I\right\vert <5\left\vert J\right\vert $.

(iii) Suppose now that $q\geq \sqrt{2}-1>2/5$. Consequently, $\left\vert
J\right\vert =1-2q<1/5$ and for any component $K$ of $N_{q}^{H}$, the
distance between $K$ and a component longer than $K$ is bigger than $%
2\left\vert K\right\vert $.

\textbf{Case 1. }$t\leq \left\vert J\right\vert $. We will show that $\sup
I<\left( 2-3q\right) +\left\vert J\right\vert $. If there is no component $%
\left[ a,b\right] $ of $N_{q}^{H}$ such that $a\leq \sup \left( t+J\right)
<b $, then $\sup I=\sup \left( t+J\right) \leq 2-3q$. If such a component
exists, denote it by $\left[ a_{1},b_{1}\right] $. Note that $%
b_{1}-a_{1}\leq q^{2}\left\vert J\right\vert $,\ since the longest component
of $N_{q}^{H}\cap \left[ 1-q,1\right] $ is not bigger than $q^{2}\left\vert
J\right\vert $. If there exists a component $\left[ a_{2},b_{2}\right] $ of $%
t+N_{q}^{H}$ such that $a_{2}\leq b_{1}<b_{2}$ then $b_{2}-a_{2}\leq q\left(
b_{1}-a_{1}\right) $. Inductively, if there exists a component $\left[
a_{n+1},b_{n+1}\right] $ such that $a_{n+1}\leq b_{n}<b_{n+1}$, then $%
b_{n+1}-a_{n+1}<q\left( b_{n}-a_{n}\right) $. If the sequence $\left( \left[
a_{n},b_{n}\right] \right) $ is finite and $\left[ a_{n_{0}},b_{n_{0}}\right]
$ is its last element, then $\sup I=b_{n_{0}}<\left( 2-3q\right)
+q\left\vert J\right\vert $. If not, we have to look more closely at the
structure of $N_{q}^{H}$. \newline
Since the set $H$ is infinite, there is a gap $\left( \alpha ,\beta \right) $
in $N_{q}^{H}\cap \left[ 0,\frac{1}{2}\left\vert J\right\vert \right] $.
Moreover, there is $n_{0}$ such that $\left\vert K_{n}\right\vert <\beta
-\alpha $ for any $n>n_{0}$. Hence 
\begin{equation*}
\sup I<b_{n_{0}}+\frac{1}{2}\left\vert J\right\vert <\left( 2-3q\right)
+\left\vert J\right\vert \text{.}
\end{equation*}%
In the same manner we check that $\inf I>\inf J-\left\vert J\right\vert $.
Finally $\sup I-\inf I<2-3q+\left\vert J\right\vert -\left( q-\left\vert
J\right\vert \right) =2-4q+2\left\vert J\right\vert =4\left\vert
J\right\vert $.

\textbf{Case 2.} $t>\left\vert J\right\vert $. In this case, using analogous
arguments as in Case 1 one can prove that $\left[ 1-q,t+q\right] \setminus
\left( N_{q}^{H}\cup (t+N_{q}^{H})\right) \neq \emptyset $. Indeed, suppose
that $1-q<t+q$. If there is no component $\left[ a,b\right] $ of $%
t+N_{q}^{H} $ such that $a\leq 1-q<b$, then $\left[ 1-q,t+q\right] \setminus
\left( N_{q}^{H}\cup (t+N_{q}^{H})\right) \neq \emptyset $. If such a
component exists, denote it by $\left[ a_{1},b_{1}\right] $. Note that $%
b_{1}-a_{1}<\frac{1}{2}\left( t+q-b_{1}\right) $. If there exists a
component $\left[ a_{2},b_{2}\right] $ of $N_{q}^{H}$ such that $a_{2}\leq
b_{1}<b_{2}$ then $b_{2}-a_{2}<q\left( b_{1}-a_{1}\right) $. Inductively, if
there exists a component $\left[ a_{n+1},b_{n+1}\right] $ such that $%
a_{n+1}\leq b_{n}<b_{n+1}$, then $b_{n+1}-a_{n+1}<q\left( b_{n}-a_{n}\right) 
$. Moreover, there is a gap in $N_{q}^{H}\cap \left[ 0,t+q-\left( q-1\right) %
\right] $. In consequence, there is a gap in $\left[ 1-q,t+q\right] \cap
\left( N_{q}^{H}\cup (t+N_{q}^{H})\right) $.\ It means that $J$\ and $t+J$\
are contained in different components of $N_{q}^{H}$. Therefore $\left\vert
I\right\vert <3\left\vert J\right\vert $\ which ends the proof.
\end{proof}

\begin{thm}
\label{mainResult} No regular Cantorval $N_{q}^{H}$, with $1\notin H$ and $%
2\in H$, is an achievement set of a sequence.
\end{thm}

\begin{proof}
Suppose that $N_{q}^{H}=E(x_{i})$ for some non-increasing sequence $(x_{i})$
with positive terms. Then, by Lemma \ref{lemat3}, we conclude that there is $%
k\in \mathbb{N}$ such that $x_{k}=q-q^{2}$ and $q^{2}=r_{k}$. Since $%
1=\sum_{j=1}^{\infty }x_{j}$, we have $1-q^{2}=1-r_{k}=x_{1}+x_{2}+...+x_{k}%
\in F_{k}$ . Consequently, $1-q=1-q^{2}-\left( q-q^{2}\right)
=x_{1}+x_{2}+...+x_{k-1}\in F_{k-1}$. Hence $F_{k}\cap \lbrack
1-q,1-q^{2})=\{1-q\}$ and $(1-q-(q-q^{2}),1-q)\cap
F_{k-1}=((1-q)^{2},1-q)\cap F_{k-1}=\emptyset $. Since the whole Cantorval $%
N_{q}^{H}=F_{k}+E_{k}$, the interval $J$ is covered by some number of the
sets $\sum_{i=1}^{k}\varepsilon _{i}x_{i}+E_{k}\subset
\sum_{i=1}^{k}\varepsilon _{i}x_{i}+[0,q^{2}]$ for $(\varepsilon _{i})\in
\{0,1\}^{k}$.

Let $m:=|F_{k}\cap (q-q^{2},2(q-q^{2})]|$. Enumerate $F_{k}\cap
(q-q^{2},2(q-q^{2})]$ as $\{y_{1}<y_{2}<\dots <y_{m}\}$. Note that $%
x_{k}+x_{k-1}\geq 2(q-q^{2})$. Therefore $y_{1},y_{2},\dots ,y_{m-1}$ are
among elements of the set $\{x_{i}:x_{i}>x_{k}\}$. The element $y_{m}$ can
be in $\{x_{i}:x_{i}>x_{k}\}$ or it can be equal to $%
x_{k}+x_{k-1}=2(q-q^{2}) $ if $x_{k}=x_{k-1}$. Now we consider three cases.

\textbf{Case 1.} $m>\frac{1}{q}$ and $\sum_{i=1}^{m}y_{i}>1-q+(2q-2q^{2})$.
Then $\sum_{i=1}^{m-1}y_{i}>1-q$. Note that $\sum_{i=1}^{m-1}y_{i}\in
F_{k-1} $. If $\sum_{i=1}^{m-1}y_{i}\in (1-q,1-q^{2})$, we have a
contradiction. So suppose that $\sum_{i=1}^{m-1}y_{i}\geq 1-q^{2}$. Let $%
A\subset \{1,2,\dots ,m-1\}$ be such that $\sum_{i\in A}y_{i}\geq 1-q^{2}$
and for any $B\subset \{1,2,\dots ,m-1\}$ such that $\sum_{i\in B}y_{i}\geq
1-q^{2}$ we have $\sum_{i\in B}y_{i}\geq \sum_{i\in A}y_{i}$. Clearly $%
\left\vert A\right\vert \geq 2$. Take any $i_{0}\in A$ and put $A^{\prime
}=A\setminus \{i_{0}\}$. Since $y_{i_{0}}<2(q-q^{2})$, then 
\begin{equation*}
\sum_{i\in A^{\prime }}y_{i}=\sum_{i\in
A}y_{i}-y_{i_{0}}>1-q^{2}-2(q-q^{2})=(1-q)^{2}.
\end{equation*}%
Note that $\sum_{i\in A^{\prime }}y_{i}\in F_{k-1}$. If $\sum_{i\in
A^{\prime }}y_{i}\neq 1-q$, we have a contradiction. If the equality $%
\sum_{i\in A^{\prime }}y_{i}=1-q$ holds, we can take another index $%
i_{1}\neq i_{0}$ instead of $i_{0}$ and we can put $A^{\prime \prime
}=A\setminus \{i_{1}\}$. Then $\sum_{i\in A^{\prime \prime }}y_{i}\in
F_{k-1} $ and $\sum_{i\in A^{\prime \prime }}y_{i}\in ((1-q)^{2},1-q)\cup
(1-q,1-q^{2})$, and we reach a contradiction as well.

\textbf{Case 2.} $m>\frac{1}{q}$ and $\sum_{i=1}^{m}y_{i}\leq
1-q+(2q-2q^{2}) $. Since $y_{1}>q-q^{2}$, 
\begin{equation*}
\sum_{i=1}^{m}y_{i}>m\left( q-q^{2}\right) >\frac{1}{q}\left( q-q^{2}\right)
=1-q\text{.}
\end{equation*}%
Fix $i_{0}\leq m$. If\ $\sum\nolimits_{i\neq i_{0}}y_{i}>1-q$ then $%
\sum\nolimits_{i\neq i_{0}}y_{i}>1-q^{2}$, because $\left(
1-q,1-q^{2}\right) \cap F_{k}=\emptyset $. Therefore%
\begin{equation*}
\sum\nolimits_{i\leq m}y_{i}=\sum\nolimits_{i\neq
i_{0}}y_{i}+y_{i_{0}}>1-q^{2}+q-q^{2}
\end{equation*}%
which contradicts the assumption $\sum_{i=1}^{m}y_{i}\leq 1-q+(2q-2q^{2})$.
Hence $\sum\nolimits_{i\neq i_{0}}y_{i}\leq 1-q$. Since $\sum\nolimits_{i%
\neq i_{0}}y_{i}>1-q^{2}$ and $\left( 1-q,1-q^{2}\right) \cap
F_{k}=\emptyset $ \ it means that $\sum\nolimits_{i\neq i_{0}}y_{i}=1-q$.%
\newline
Thus we have proved that $\sum\nolimits_{i\neq i_{0}}y_{i}=1-q$ for any $%
i_{0}\leq m$. It gives a contradiction because $y_{i}$ are different from
each other.

\textbf{Case 3.} $m\leq \frac{1}{q}$. We will consider separately $N_{q}^{H}$
for $q>1/3$ and for $q\leq 1/3$. Suppose that $q>1/3$. It means that $m\leq
2 $. The assumption that $m=1$ immediately leads to a contradiction. Thus
there exist exactly two elements $x_{k-1},x_{k-2}\in (q-q^{2},2(q-q^{2})]$.
Note that for $q\in \left[ 1/3,1\right] $ the inequality $\left( 1-q\right)
^{2}\leq 2\left( q-q^{2}\right) $ holds. Moreover, $((1-q)^{2},1-q)\cap
F_{k-1}=\emptyset $. Therefore there are exactly two elements of the
sequence $\left( x_{i}\right) $ in the interval $(q-q^{2},1-q)$. It follows
that%
\begin{equation*}
(q,1-q)\subset \left( x_{k-1}+\left( N_{q}^{H}\cap \left[ 0,x_{k}\right]
\right) \right) \cup \left( x_{k-2}+\left( N_{q}^{H}\cap \left[ 0,x_{k}%
\right] \right) \right) \text{.}
\end{equation*}%
Note that the length of the longest component of $N_{q}^{H}\cap \left[
0,x_{k}\right] $\ is at most $q^{2}\left( 1-2q\right) $. If $q<\sqrt{2}-1$\
then $q^{2}<0.42^{2}<1/5$\ which gives a contradiction with Lemma \ref%
{MainLemma}(ii). For $q\geq \sqrt{2}-1$\ we use the inequality $q^{2}<1/4$\
and obtain a contradiction with Lemma \ref{MainLemma}(iii).\newline
Suppose now that $q\leq 1/3$ and $x_{k}=x_{k-1}$. Then $%
y_{m}=x_{k}+x_{k-1}=2q-2q^{2}$. First assume that $n+2\notin H$ and put $%
H^{\prime \prime }=\left( H-2\right) \cap \mathbb{N}$. By Lemma \ref%
{MainLemma}(i) the longest interval $I$ contained in the translations $%
\bigcup_{i=1}^{m-1}(y_{i}+q^{-n+3}N_{q}^{H^{\prime \prime }})$ has the
length smaller than $q^{2}((m-1)(1-2q)+(m-2)q^{2})$. If the left endpoint of 
$I$ equals $q$, then the right endpoint of $I$ is smaller than 
\begin{equation*}
q+q^{2}((m-1)(1-2q)+(m-2)q^{2})=q+q^{2}(m(1-2q)+(m-1)q^{2})-q^{2}(1-q)^{2}
\end{equation*}%
\begin{equation*}
=q+q^{2}(m(1-2q+q^{2})-q^{2})+q^{2}(1-q)^{2}\leq
q+q(1-2q+q^{2})-q^{4}-q^{2}(1-q)^{2}
\end{equation*}%
\begin{equation*}
\leq 2q-2q^{2}+q^{3}-q^{2}(1-q)^{2}.
\end{equation*}%
On the other hand, the next translate $y_{m}+q^{-n+3}N_{q}^{H^{\prime \prime
}}$ cannot cover interval $[2q-2q^{2},2q-2q^{2}+q^{3}]$. Therefore the union
of all translations cannot cover the interval $%
[2q-2q^{2}+q^{3}-q^{2}(1-q)^{2},2q-2q^{2}+q^{3}]$, and we reach a
contradiction. Note that if $n+2\in H$, then the estimation is even sharper.

Finally, suppose that $q\leq 1/3$, $m\leq \frac{1}{q}$ and $x_{k}<x_{k-1}$.
Then $y_{m}$ is one of $x_{i}$. Assume as before that $n+2\notin H$. By
Lemma \ref{MainLemma}(i) the longest interval $I$ contained in the
translations $\bigcup_{i=1}^{m}(y_{i}+q^{-n+3}N_{q}^{H^{\prime \prime }})$
has the length smaller than $q^{2}(m(1-2q)+(m-1)q^{2})$. If the left
endpoint of $I$ equals $q$, then the right endpoint of $I$ is smaller than $%
2q-2q^{2}+q^{3}$. Since $x_{k}<x_{k-1}$, then $x_{k}+x_{k-1}>2q-2q^{2}$ and
the translate $x_{k}+x_{k-1}+q^{-n+3}N_{q}^{H^{\prime \prime }}$ cannot
cover the interval $[2q-2q^{2}+q^{3},x_{k}+x_{k-1}+q^{3}]$. As before, we
reach a contradiction.
\end{proof}

\begin{cor}
No regular Cantorval $N_{q}^{H}$, with $1\notin H$ and $2\in H$, is
isometric to an achievement set of a sequence.
\end{cor}

In fact, any set $E$ isometric to the symmetric set of the form $N_{q}^{H}$
is of the form $E=N_{q}^{H}+t$ for some $t\in \mathbb{R}$. If $E=E\left(
x_{n}\right) $ for an absolutely summable sequence $\left( x_{n}\right) $,
then $0\in E$ and hence $t\leq 0$. More precisely, $t$ is the sum of all
negative terms of $\left( x_{n}\right) $. Then (see for example \cite{BFPW}) 
$N_{q}^{H}=E\left( \left\vert x_{n}\right\vert \right) $ which gives the
contradiction with our theorem.

\begin{rem}
In the previous theorem it is sufficient to assume that $n-1\in H$, $n\notin
H$ and $n+1\in H$ for some $n$ (instead of the assumption $1\notin H$ and $%
2\in H$). In fact, let us consider the part $\widetilde{N}=\left[ 0,q^{n-1}%
\right] \cap N_{q}^{H}$ of the set $N_{q}^{H}$, with nonremoved middle
closed interval $J$ and with removed next intervals $G_{0}$ and $G_{1}$
lying on the left and right of $J$, respectively. To simplify the notation,
let us assume (by considering rescaling $q^{-n+1}N_{q}^{H}$ of $N_{q}^{H}$)
that the right end of $N$ is equal to one. Then $J=[q,1-q]$, $%
G_{0}=(q^{2},q-q^{2})$ and $G_{1}=(1-q+q^{2},1-q^{2})$. Because $n-1\in H$,
one is the left end of the gap which is the longest one from the left.
Hence, by the assumption that $N_{q}^{H}=E\left( x_{n}\right) $ and by Lemma %
\ref{lemat3}, we have $1=r_{z}$ for some $z\in \mathbb{N}$, and consequently 
$\widetilde{N}=E\left( x_{n}:n>z\right) $. So we can repeat the conclusion
that $F_{k}\cap \lbrack 1-q,1-q^{2})=\{1-q\}$ and $(1-q-(q-q^{2}),1-q)\cap
F_{k-1}=((1-q)^{2},1-q)\cap F_{k-1}=\emptyset $. The rest of the proof runs
as before.

It is also worth noting that if $q\leq 1/3$ then we do not need any
assumptions on $H$. Indeed, since $N_{q}^{H}$ is regular, there exists $n$
such that $n\notin H$ and $n+1\in H$. By the inequality $q\leq 1/3$ and the
assumption $N_{q}^{H}=E\left( x_{n}\right) $, we easily conclude that if $%
q-q^{2}=x_{k}$ then $k=p$ in the notation from the Second Gap Lemma (Lemma %
\ref{second_gap_lemma}). Therefore, by this Lemma, we have $1-q^{2}\in F_{k}$
and $1-q+q^{2}\in F_{k}+r_{k}$. Since $(1-q+q^{2})-(1-q)=q^{2}=r_{k}$, we
have $1-q\in F_{k}$. By $(1-q^{2})-(1-q)=q-q^{2}=x_{k}$ we conclude that in
fact $1-q\in F_{k-1}$. This way we get once again $F_{k}\cap \lbrack
1-q,1-q^{2})=\{1-q\}$ and $((1-q)^{2},1-q)\cap F_{k-1}=\emptyset $.
\end{rem}

\section{Attractors of IFS}

Let us recall that all known examples of sequences generating Cantorvals as
their achievement sets, are so called multigeometric sequences. We call a
sequence multigeometric if it is of the form 
\begin{equation*}
(k_{1},k_{2},\ldots ,k_{m},k_{1}q,k_{2}q,\ldots ,k_{m}q,k_{1}q^{2},\ldots )
\end{equation*}%
for some $k_{1},k_{2},\ldots ,k_{m},q\in \mathbb{R}$. Following \cite{BFS},
we denote such a sequence by $(k_{1},k_{2},\ldots ,k_{m};q)$. It was
observed in \cite{BBFS}, that for the multigeometric sequence $%
(a_{n})=(k_{1},k_{2},\ldots ,k_{m};q)$ its achievement set $E$ is the unique
compact set satisfying the equality $E=\Sigma +qE$ for $\Sigma
=\{\sum_{n=1}^{m}\varepsilon _{n}k_{n}:(\varepsilon _{n})\in \{0,1\}^{m}\}$.
It is equivalent to the fact that $E$ is an attractor of the IFS consisting
of functions $\{f_{\sigma }:\sigma \in \Sigma \}$ given by formulas $%
f_{\sigma }(x)=qx+\sigma $. On the other hand, even symmetric set $\Sigma $
does not need to be connected with some multigeometric sequence.

In this section we also need to briefly present the notion of IFS. Let $%
(X,d) $ be a complete metric space and denote by $K(X)$ space of all
nonempty compact subsets of $X$ endowed with the Hausdorff metric $H$: 
\begin{equation*}
H(A,B)=\max \left\{ \sup \left\{ \inf \{d(x,y):y\in A\}:x\in B\right\} ,\sup
\left\{ \inf \{d(x,y):y\in B\}:x\in A\right\} \right\} .
\end{equation*}%
It turns out that $(K(X),H)$ is complete if and only if $X$ is complete.

\begin{df}
If $f_{1},...,f_{n}:X\rightarrow X$ are functions with Lipschitz constants
strictly smaller than $1$, then the finite sequence $\mathcal{S}%
=(f_{1},...,f_{n})$ is called an iterated function system (IFS in short).%
\newline
If $\mathcal{S}=(f_{1},...,f_{n})$ is an IFS, then by $F_{\mathcal{S}}$ we
denote the mapping $F_{\mathcal{S}}:K(X)\rightarrow K(X)$ defined by: 
\begin{equation*}
F_{\mathcal{S}}(D)=f_{1}(D)\cup ...\cup f_{n}(D)\text{.}
\end{equation*}%
It turns out that $F_{\mathcal{S}}$ also has Lipschitz constant smaller than 
$1$ thus, by Banach Fixed Point Theorem, it has the unique fixed point $%
D_{0} $. We call it an attractor of an IFS $\mathcal{S}$.

This nice idea has been deeply considered in the last 30 years (see, a.e., 
\cite{H}, \cite{Ha}, \cite{B}).\newline
\end{df}

\begin{thm}
\label{IFSTheorem} For any $q\leq 1/3$ the set $N_{q}=N_{q}^{2\mathbb{N}}$
is the attractor of IFS consisting of $4+\lceil \frac{1}{q^{2}}\rceil $
affine functions given by the formulas 
\begin{equation*}
f_{\sigma }(x)=q^{2}x+\sigma
\end{equation*}%
for $\sigma $'s belonging to some finite set $\Sigma $. (The symbol $\lceil
\alpha \rceil $ denotes the smallest integer greater than or equal to $%
\alpha $.)
\end{thm}

\begin{proof}
For the sake of clarity let us start with the proof for $q=1/3$. Let $\Sigma
=\frac{1}{27}\{0,6,8,9,10,\ldots ,16,18,24\}$. It is enough to verify that $%
N_{1/3}=\Sigma +\frac{1}{9}N_{1/3}$. Instead of $N_{1/3}$ we can consider
the set $\tilde{N}=27N_{1/3}$ and $\widetilde{\Sigma }=27\Sigma $.

\includegraphics[height=4cm,width=16cm]{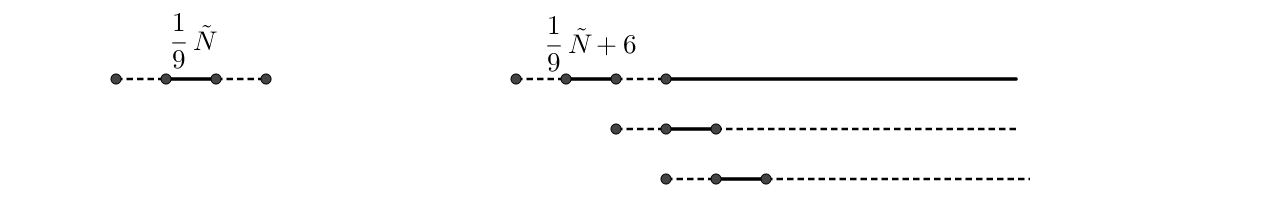}\newline
Figure 3

Observe that (see Figure 3):

\begin{itemize}
\item $[0,3] \cap \tilde{N} = \frac{1}{9} \tilde{N}$,

\item $[6,9] \cap \tilde{N} = \frac{1}{9} \tilde{N}+6$,

\item $[8,9]\cap \tilde{N}=(\frac{1}{9}\tilde{N}+8)\cap \lbrack 8,9]=(\frac{1%
}{9}\tilde{N}+6)\cap \lbrack 8,9]$,

\item The intervals $[9,10],\ [10,11],\ldots ,\ [17,18]$ are the middle
components of $\frac{1}{9}\tilde{N}+8,\frac{1}{9}\tilde{N}+9,\ldots ,\frac{1%
}{9}\tilde{N}+16$ respectively,

\item The situation between $18$ and $27$ is symmetric and analogous to the
one between $0$ and $9$.
\end{itemize}

Now assume that $q<1/3$. Then $N_{q}$ is the union of four translations of
the set $q^{2}N_{q}$ and of the middle interval $J$. Denote by $L$ the union
of two left copies of $q^{2}N_{q}$, and by $R$ the union of two right ones.

The images of the set $N_{q}$ with respect to the functions%
\begin{equation*}
f_{1}(x)=q^{2}x\text{ \ and \ }f_{2}(x)=q^{2}x+q-q^{2}
\end{equation*}%
cover the set $L$, and the images of the set $N_{q}$ with respect to the
functions 
\begin{equation*}
f_{3}(x)=q^{2}x+1-q\text{ \ and \ }f_{4}(x)=q^{2}x+1-q^{2}
\end{equation*}%
cover the set $R$. Now we want to cover the interval $J$ by the images of $J$
with respect to the functions $g_{i}$ given by the formulas $%
g_{i}(x)=q^{2}x+\sigma _{i}$, for $i=1,2,...,m$.

\includegraphics[height=3cm,width=16cm]{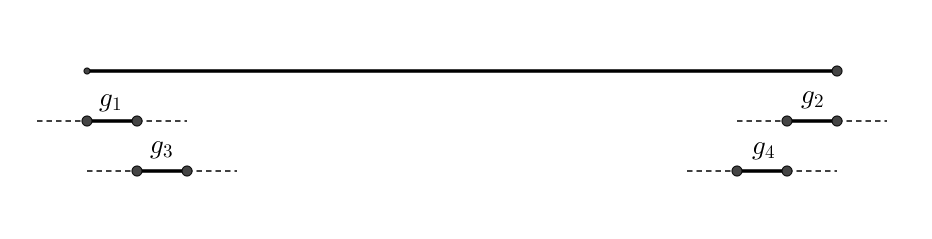}\newline
Figure 4

We choose $\sigma _{1}$ such that $g_{1}\left( q\right) =q$, and $\sigma
_{2} $ such that $g_{2}\left( 1-q\right) =1-q$. It is worth noting that $%
g_{1}\left[ L\right] $ is consistent with $f_{2}\left[ R\right] $, and $g_{2}%
\left[ R\right] $ is consistent with $f_{3}\left[ L\right] $. The images of $%
q^{2}N_{q}$ with respect to next functions cover $J$ one by one from the
left and from the right. Note that the "middle" two subintervals can be
overlapping.
\end{proof}

\begin{rem}
In \cite{BBFS} the authors have analysed the attractor of IFS for the given
set $\Sigma $ with respect to the ratio $q$. Making a simple use of their
results we easily conclude that for $\Sigma $ described in the proof of
Theorem \ref{IFSTheorem} for the Cantorval $N_{1/3}$\ , the attractor $%
E(\Sigma ,q)$ is

\begin{itemize}
\item an interval for $q \geq \frac{1}{5}$,

\item a Cantorval for $\frac{1}{9}\leq q < \frac{1}{5}$,

\item the set with positive Lebesgue measure for almost all $q\in (\frac{1}{%
13},\frac{1}{9})$. Unfortunately, we can not say anything on the interior of 
$E(\Sigma ,q)$ in this case,

\item measure zero set for some sequence $(q_{n})$ decreasing to $\frac{1}{13%
}$,

\item a nowhere dense set for $q\leq \frac{1}{13}$.
\end{itemize}
\end{rem}

For $q>1/3$ the situation is more complicated.

\begin{thm}
Let $q\in \left( 1/3,1/2\right) $. Then the set $N_{q}$ is the attractor of
IFS consisting of affine functions with three different ratios.

\begin{proof}
We start from the four functions $f_{i}$ exactly as in the proof of the
previous theorem. Now we choose a positive integer $n$ satisfying the
inequality $q^{2n}\left( 1-q\right) \leq 1-2q$. Then some translation of $%
q^{2n}\left( J\cup R\right) $ is contained in $J$. We choose $\sigma _{1}$
and $\sigma _{2}$ such that $g_{i}\left( x\right) =q^{2n}x+\sigma _{i}$, $%
g_{1}\left[ J\right] $ covers the left and $g_{2}\left[ J\right] $ covers
the right subinterval of $J$. By the self-similarity of the set $N_{q}$ we
observe that $g_{1}\left[ L\right] $ is consistent with $f_{2}\left[ R\right]
$, and $g_{2}\left[ R\right] $ is consistent with $f_{3}\left[ L\right] $.

Finally, let $k$ be an integer satisfying the inequalities

\begin{itemize}
\item $q^{2k}\cdot q\leq q^{2n}\left( 1-2q\right) ,$

\item $q^{2k}\left( 1-2q\right) \leq \left( 1-2q\right) -2\left(
q^{2n}\left( 1-2q\right) \right) $.
\end{itemize}

Then we can cover the middle part of $J$ by the images of $J$ with respect
to several functions $h_{i}$ of the form $h_{i}\left( x\right)
=q^{2k}x+\theta _{i}$, similarly as in the previous proof.
\end{proof}
\end{thm}

\begin{problem}
In the last theorem we have used three ratios. It is the natural question if
it is possible to use only one ratio. Does there exist a set $N_{q}$ with $%
q>1/3$ which is the attractor of IFS consisting of affine functions with the
same ratios?
\end{problem}


\begin{thebibliography}{99}
\bibitem{BBFS} T. Banakh, A. Bartoszewicz, M. Filipczak, E. Szymonik \emph{%
Topological and measure properties of some self-similar sets}, Topol.
Methods Nonlinear Anal., \textbf{46} (2015), 1013--1028.

\bibitem{BBGS} T. Banakh, A. Bartoszewicz, Sz. G{\l }\c{a}b, E. Szymonik, 
\emph{Algebraic and topological properties of some sets in }$\ell _{1}$,
Colloq. Math., \textbf{129(1)} (2012), 75--85.

\bibitem{BPW} M. Banakiewicz, F. Prus-Wi\'{s}niowski, \emph{M-Cantorvals of
Ferens type}, to appear in Mathematica Slovaca 67 (2017).

\bibitem{B} M. F. Barnsley, \emph{Fractals everywhere}. Academic Press
Professional, Boston, MA, 1993.

\bibitem{BFPW} A. Bartoszewicz, M. Filipczak, F. Prus-Wi\'{s}niowski, \emph{%
Topological and algebraic aspects of subsums of series}, chapter in \emph{%
Traditional and present-day topics in real analysis}, \L \'{o}d\'{z}
University Press (2013), 345--366.

\bibitem{BFS} A. Bartoszewicz, M. Filipczak, E. Szymonik, \emph{%
Multigeometric sequences and Cantorvals}, Cent. Eur. J. Math., \textbf{12(7)}
(2014), 1000--1007.

\bibitem{BPW} W. Bielas, S. Plewik, M. Walczy\'{n}ska, \emph{On the center
of distances}, arXiv:1605.03608v2.

\bibitem{F} C. Ferens, \emph{On the range of purely atomic probability
measures}, Studia Math., \textbf{77(3)} (1984), 261--263.

\bibitem{GN} J. A. Guthrie, J. E. Nymann, \emph{The topological structure of
the set of subsums of an infinite series}, Colloq. Math., \textbf{55(2)}
(1988), 323--327.

\bibitem{Ha} M. Hata, \emph{On the structure of self-similar sets.} Japan J.
Appl. Math. 2 (1985), 381--414.

\bibitem{H} J. Hutchinson, \emph{Fractals and self-similarity}. Indiana
Univ. Math. J. \textbf{30} (1981), no 5, 713--747.

\bibitem{J} R. Jones, \emph{Achievement sets of sequences}, Amer. Math.
Monthly \textbf{77(3)} (2011), 508--521.

\bibitem{K} S. Kakeya, \emph{On the partial sums of an infinite series}, The
Science Reports of Tohoku University, \textbf{3} (1914) 159--164.

\bibitem{MO} P. Mendes, F. Oliveira, \emph{On the topological structure of
the arithmetic sum of two Cantor sets}, Nonlinearity \textbf{7(2)} (1994)
329--343.

\bibitem{N1} Z. Nitecki, \emph{The subsum set of a null sequence},
arXiv:1106.3779v1.

\bibitem{N2} Z. Nitecki, \emph{Cantorvals and subsum sets of null sequences}%
, Amer. Math. Monthly, \textbf{122} (2015), no. 9, 862--870.

\bibitem{NS0} J. E. Nymann, R. A. Saenz, \emph{On a paper of Guthrie and
Nymann on subsums on infinite series}, Colloq. Math., \textbf{83(1)} (2000),
1--4.

\bibitem{WS} A. D. Wainshtein, B. Z. Shapiro, \emph{Structure of a set of $%
\overline{\alpha }$-representable numbers}, Izv. Vyssh. Uchebn. Zaved. Mat., 
\textbf{5} (1980), 8--11 (in Russian).
\end{thebibliography}
\end{document}